\documentclass[12pt, twoside, leqno]{article}
\usepackage{amsmath,amsthm}
\usepackage{amssymb,latexsym}
\usepackage{enumerate}
\newtheorem{thm}{Theorem}[section]
\newtheorem{lem}[thm]{Lemma}
\newtheorem{prop}[thm]{Proposition}
\newtheorem{cor}[thm]{Corollary}
\theoremstyle{definition}
\newtheorem{defin}[thm]{Definition}
\newtheorem*{xrem}{Remark}%IMP
\newtheorem{mainthm}[thm]{Main Theorem}
\numberwithin{equation}{section}
\begin{document}
\baselineskip=17pt
\title{A second-order identity for the Riemann tensor and applications}
\author{Carlo Alberto Mantica and 
Luca Guido Molinari\\
Physics Department, Universit\'a degli Studi di Milano\\
Via Celoria 16, 20133 Milano, Italy\\
E-mail: luca.molinari@mi.infn.it}
\date{June 2009}
\maketitle
\renewcommand{\thefootnote}{}
\footnote{2000 \emph{Mathematics Subject Classification}:
Local Riemannian geometry 53B20, Methods of Riemannian Geometry 53B21.}
\footnote{\emph{Key words and phrases}: Riemann tensor, symmetric manifold, 
recurrent manifold, weakly-Ricci symmetric manifold.}
\renewcommand{\thefootnote}{\arabic{footnote}}
\setcounter{footnote}{0}
%\keywords{Riemann tensor, symmetric manifold, recurrent manifold, weakly-Ricci
%symmetric manifold}
%\subjclass[2000]{Local Riemannian geometry 53B20, Methods of Riemannian 
%Geometry 53B21}

\begin{abstract}
A second-order differential identity for the Riemann
tensor is obtained, on a manifold with symmetric connection.
Several old and some new differential identities for the Riemann and 
Ricci tensors descend from it. Applications to
manifolds with Recurrent or Symmetric structures are discussed.
The new structure of $K$-recurrency naturally emerges from an invariance
property of an old identity by Lovelock.
\end{abstract}

\section{Introduction}\label{sect1}
Given a symmetric connection %$\Gamma_{ab}^c=\Gamma_{ba}^c$ 
on a smooth manifold, one introduces the covariant derivative %$\nabla_a$ 
and the Riemann curvature tensor
%\begin{displaymath}
$R_{abc}{}^d = \partial_a \Gamma_{bc}^d - \partial_b\Gamma_{ac}^d -
\Gamma_{ac}^k\Gamma_{bk}^d + \Gamma_{ak}^d \Gamma_{bc}^k $.
%\end{displaymath}
The tensor is antisymmetric in $a,b$ and satisfies the two
Bianchi identities, $ R_{(abc)}{}^d = 0 $ and $\nabla_{(a}R_{bc)d}{}^e =0$.  
%\label{IBianchi}\\
%\nabla_{(a}R_{bc)d}{}^e &=0 
%\begin{align}
%R_{(abc)}{}^d &=0 , 
%\label{IBianchi}\\
%\nabla_{(a}R_{bc)d}{}^e &=0 
%\label{IIBianchi}
%\end{align}
\footnote{
Hereafter the symbol ($\cdots $) denotes a summation over {\emph cyclic} 
permutation of tensor indices:
$K_{(abc)}$ $=K_{abc}+K_{bca}+K_{cab}$. If 
$K_{abc}$ $=-K_{bac}$ then $K_{(abc)}=3K_{[abc]}$,
where $[abc]$ means complete antisymmetrization.}
% (divided by 3!). 
%This is the case of Bianchi identities, which are usually written in such 
%form. 
From the Bianchi identities various others for the Riemann tensor and 
the Ricci tensor $R_{ac}=R_{abc}{}^b$ can be derived. 
The following first-order one is due to 
Oswald Veblen \cite{Eisenhart,Lovelock}:
\begin{equation}
\nabla_a R_{bcd}{}^e-\nabla_b R_{adc}{}^e+\nabla_c R_{adb}{}^e
-\nabla_d R_{bca}{}^e =0 \label{veblenequation}
\end{equation}
%Contraction of Veblen's formula gives $\nabla_{[a}R_{bc]}=0$.
If the connection is inherited from a metric, Walker's identity of second 
order holds \cite{Walker51,Schouten}
\begin{equation}
[\nabla_a,\nabla_b]R_{cdef}+[\nabla_c,\nabla_d]R_{abef}+[\nabla_e,\nabla_f]
R_{abcd}=0\label{walker}
\end{equation}
and, if the Ricci tensor vanishes, Lichnerowicz's non linear wave 
equation holds \cite{Hughston,Misner} 
\begin{equation}
\nabla^e\nabla_e R_{abcd} + R_{ab}{}^{ef}R_{efcd}
 - 2 R^e{}_{ac}{}^f R_{ebdf} + 2R_e{}_{ad}{}^f R^e{}_{bcf}=0 
\label{Lichnerowicz}
\end{equation}

In this paper we derive, with the only requirement that the connection is 
symmetric, a useful identity for the cyclic combination 
$\nabla_{(a}\nabla_bR_{cd)e}{}^f $.
%, that appears to be an extension of the Bianchi identities to the next order.
An identity by Lovelock for the divergence of the 
Riemann tensor follows. 
We show that it holds more generally for curvature tensors $K$
originating from the Riemann tensor  (Weyl, Concircular etc.).

The main identity and Lovelock's enable us to reobtain in a unified
manner various known identities, and some new ones, that apply in
Riemannian spaces with structures. In Section 3 we show that various
differential structures, such as a) Locally symmetric, 
b) Nearly Conformally Symmetric,
c) Semisymmetric, d) Pseudosymmetric, e) Generalized Recurrent, lead 
to the same set of algebraic identities for the Riemann tensor. 
We then introduce the new structures, $K$-harmonic and $K$-recurrent,
that also yield the set of identities, and include cases a), b), e)
and others, that arise from different choices of the $K-$tensor.

In Section 4, the Weakly Ricci Symmetric structure is considered,
with its covectors A,B,D. We show that one of the above algebraic 
identities holds iff the vector field $A-B$ is closed.
We end with Section 5, where we derive Lichnerowicz's wave equation 
for the Riemann tensor from the main equation.

\section{A second order identity}\label{sect2}
We begin with the main identity; as a corollary we reobtain an identity 
by Lovelock which is used throughout the paper, and a generalization of it, 
for $K-$curvature tensors.
\begin{mainthm}{(The second order identity)}\\
In a smooth differentiable manifold with symmetric connection: 
\begin{equation}
\nabla_{(a}\nabla_bR_{cd)e}{}^f = R_{(abc}{}^m R_{d)me}{}^f
-R_{ace}{}^m R_{bdm}{}^f + R_{acm}{}^f R_{bde}{}^m 
\label{iii}
\end{equation}
\end{mainthm}
\begin{proof}
Take a covariant derivative of the second Bianchi identity, and sum
over ciclic permutations of four indices $abcd$:
\begin{align}
0&=\nabla_a \nabla_{(b}R_{cd)e}{}^f + \nabla_b \nabla_{(c}R_{da)e}{}^f 
+\nabla_c \nabla_{(d}R_{ab)e}{}^f + \nabla_d \nabla_{(a}R_{bc)e}{}^f
\nonumber \\ 
&=2\nabla_{(a}\nabla_b R_{cd)e}{}^f +[\nabla_b,\nabla_a]R_{cde}{}^{f}
        +[\nabla_c,\nabla_b]R_{dae}{}^{f}\nonumber \\
&\quad +[\nabla_d,\nabla_c]R_{abe}{}^{f}
        +[\nabla_a,\nabla_d]R_{bce}{}^{f}
 +[\nabla_a,\nabla_c]R_{dbe}{}^{f}
        +[\nabla_b,\nabla_d]R_{ace}{}^{f}\label{commut}
\end{align}
The action of a commutator on the curvature tensor gives quadratic terms
\begin{displaymath}
[\nabla_a,\nabla_b]R_{cde}{}^{f} =
-R_{abc}{}^kR_{kde}{}^f
 - R_{abd}{}^kR_{cke}{}^f -R_{abe}{}^kR_{cdk}{}^f+ R_{abk}{}^f R_{cde}{}^k
\end{displaymath}
that produce 24 quadratic terms. Eight of them cancel 
because of the antisymmetry of $R$ and the remaining ones can be
grouped as follows:
\begin{align}
&=2\nabla_{(a}\nabla_b R_{cd)e}{}^f +  2R_{ace}{}^s R_{bds}{}^f +
2R_{acs}{}^f R_{dbe}{}^s\nonumber \\
&-R_{sce}{}^f (R_{adb}{}^s + R_{bda}{}^s +R_{abd}{}^s)
-R_{sbe}{}^f (R_{dca}{}^s + R_{acd}{}^s +R_{dac}{}^s)\nonumber\\
&-R_{ase}{}^f (R_{dcb}{}^s + R_{bdc}{}^s +R_{bcd}{}^s)
-R_{dse}{}^f (R_{cba}{}^s + R_{acb}{}^s +R_{abc}{}^s).\nonumber
\end{align}
The last two lines simplify by the first Bianchi identity,
\begin{align}
&=2\nabla_{(a}\nabla_b R_{cd)e}{}^f +  2R_{ace}{}^s R_{bds}{}^f +
2R_{acs}{}^f R_{dbe}{}^s
-2R_{sce}{}^f R_{adb}{}^s + 2R_{sbe}{}^f R_{cda}{}^s \nonumber\\
&-2R_{ase}{}^f R_{bcd}{}^s +2R_{sde}{}^f R_{abc}{}^s \nonumber
\end{align}
Four terms are seen to be a cyclic summation $(abcd)$.
\end{proof}

The contraction of $f$ with the index $e$ gives an equation for the 
antisymmetric part of the Ricci tensor  
(which coincides with $R_{abc}{}^c$ by the first Bianchi identity):
\begin{cor}
If $U_{ab}= R_{ab}-R_{ba}$ then  
 $\nabla_{(a}\nabla_bU_{cd)} = R_{(abc}{}^mU_{d)m}$.
\end{cor}

The contraction of $f$ with the index $a$ brings to an identity for the 
{\em divergence} of the Riemann tensor (\cite{Lovelock} ch. 7), which 
will be used extensively. We refer to it as 
\begin{cor} {(Lovelock's differential identity)}
\begin{align}
\nabla_a\nabla_m R_{bce}{}^m + \nabla_{b}\nabla_m R_{cae}{}^m+
\nabla_c\nabla_m R_{abe}{}^m\nonumber \\
= R_{am}R_{bce}{}^m + R_{bm}R_{cae}{}^m + R_{cm}R_{abe}{}^m  \label{DdivR}
\end{align}
\end{cor}
\begin{proof} The contraction in (\ref{iii}) gives 
$\nabla_{(a}\nabla_bR_{cd)e}{}^a = 
R_{(abc}{}^m R_{d)me}{}^a-R_{ace}{}^m R_{bdm}{}^a - R_{cm} R_{bde}{}^m$. 
The two cyclic sums are now written explicitly:
$ \nabla_{a}\nabla_bR_{cde}{}^a
 +\nabla_{b}\nabla_cR_{de}
 -\nabla_{c}\nabla_dR_{be}
 +\nabla_{d}\nabla_aR_{bce}{}^a = 
 R_{abc}{}^m R_{dme}{}^a
-R_{bcd}{}^m R_{me} 
+R_{cda}{}^m R_{bme}{}^a 
+R_{dab}{}^m R_{cme}{}^a 
-R_{ace}{}^m R_{bdm}{}^a 
-R_{cm} R_{bde}{}^m $.
Next the order of covariant derivatives is exchanged, in the first term of 
the l.h.s. Some terms just cancel and a triplet vanishes for a Bianchi 
identity. One gets 
$\nabla_b\nabla_a R_{cde}{}^a+
\nabla_{b}\nabla_cR_{de}-\nabla_{c}\nabla_dR_{be}
+\nabla_{d}\nabla_aR_{bce}{}^a =  R_{ba}R_{cde}{}^a -R_{ae}R_{bcd}{}^a
-R_{ca}R_{bde}{}^a$.
A Ricci term in the l.h.s. is replaced with the identity
$\nabla_c \nabla_d R_{be}= \nabla_c(\nabla_b R_{de} -\nabla_a
R_{dbe}{}^a)$. The l.h.s. becomes:
$\nabla_b\nabla_a R_{cde}{}^a+ \nabla_c\nabla_a R_{dbe}{}^a +  
\nabla_{d}\nabla_aR_{bce}{}^a + [\nabla_b,\nabla_c]R_{de}$
It is a cyclic sum on $(bcd)$ plus a commutator. The latter
is moved to the r.h.s. and evaluated. A cancellation of two
terms occurs and the r.h.s. ends as a cyclic sum too:  
$R_{bce}{}^a R_{da}+R_{cde}{}^a R_{ba}+ R_{dbe}{}^a R_{ca}.$
\end{proof}

\begin{xrem} Lovelock's identity is left invariant if the divergence of 
the Riemann tensor in the l.h.s. is replaced by the divergence of any 
curvature tensor $K$ with the property
\begin{equation}
\nabla_m K_{bce}{}^m = A\, \nabla_m R_{bce}{}^m + B\,(a_{be}\nabla_c \varphi-
a_{ce}\nabla_b\varphi), \label{divQ}
\end{equation}
where $A$ and $B$ are nonzero constants, $\varphi $ is a real 
scalar function and 
$a_{bc}$ is a symmetric (0,2) Codazzi tensor, i.e. $\nabla_b a_{cd}=\nabla_c
a_{bd}$ \cite{Derdzinski83}.
\end{xrem}
Some curvature tensors $K$ with the property (\ref{divQ}) and 
trivial Codazzi tensor (i.e. constant multiple of the metric)
are well known: Weyl's conformal tensor $C$
 \cite{Postnikov}, the projective curvature tensor $P$ \cite{Eisenhart}, 
the concircular tensor $\tilde C$ \cite{Yano40,Schouten}, the
conharmonic tensor $N$ \cite{Mishra84,Singh99} and the quasi conformal 
curvature tensor $W$ \cite{Yano68}. Their definitions and some 
identities used
in this paper are collected in the Appendix.
Since in the next section we introduce the concept of $K$-{\em recurrency}, 
and Weyl's tensor will be considered in section \ref{weaklysymm},
we give a proof of the remark:
\begin{prop}\label{propdivQ}
\begin{align}
&&\nabla_a\nabla_m K_{bce}{}^m + \nabla_{b}\nabla_m K_{cae}{}^m+
\nabla_c\nabla_m K_{abe}{}^m\nonumber \\
&&= A[R_{am}R_{bce}{}^m + R_{bm}R_{cae}{}^m + R_{cm}R_{abe}{}^m ] \label{DdivQ}
\end{align}
\end{prop}
\begin{proof}
The covariant derivative $\nabla_a$ of (\ref{divQ}) is evaluated and 
then summed with indices chosen as in Lovelock's identity. 
Since a symmetric connection is assumed, we obtain:
\begin{eqnarray}
&&\nabla_a\nabla_m K_{bce}{}^m + \nabla_{b}\nabla_m K_{cae}{}^m+
\nabla_c\nabla_m K_{abe}{}^m\nonumber \\
&&=A[\nabla_a\nabla_m R_{bce}{}^m + \nabla_{b}\nabla_m R_{cae}{}^m+
\nabla_c\nabla_m R_{abe}{}^m]\nonumber\\
&& + B[(\nabla_b a_{ce}-\nabla_c a_{be})\nabla_a\varphi+
(\nabla_c a_{ae}-\nabla_a a_{ce})\nabla_b\varphi +
(\nabla_a a_{be}-\nabla_b a_{ae})\nabla_c\varphi ].\nonumber
\end{eqnarray}
The last line is zero if $a_{bc}$ is a Codazzi tensor.
Lovelock's identity is then used to write the r.h.s. as in (\ref{DdivQ}).
\end{proof}
An apparently new Veblen-type identity for the divergence of the 
Riemann tensor can be obtained by summing Lovelock's identity
with indices cycled:
\begin{cor} 
\begin{eqnarray}\label{divVeblen}
&&\nabla_a\nabla_m R_{bec}{}^m - 
\nabla_b\nabla_mR_{ace}{}^m + \nabla_c\nabla_m R_{eba}{}^m -
\nabla_e\nabla_m R_{cab}{}^m\nonumber\\
&& = R_{am} R_{bec}{}^m 
-R_{bm}R_{ace}{}^m + R_{cm} R_{eba}{}^m -R_{em} R_{cab}{}^m 
\end{eqnarray}
\end{cor}
\begin{proof}
Write Lovelock's identity (\ref{DdivR}) for all cyclic permutations
of $(a,b,c,e)$ and sum them. Simplify by using the first Bianchi identity. 
\end{proof}
We note that an analogous Veblen-type identity can be obtained
for a tensor $K$, starting from Proposition (\ref{propdivQ}). 
\begin{cor}
In a manifold with Levi-Civita connection 
\begin{align}
\nabla_{m}\nabla_{n}R_{ab}{}^{mn} = 0\label{divdivR}
\end{align}
\end{cor}
\begin{proof}
Eq.(\ref{DdivR}) is contracted with $g^{ce}$. The formula is reported in 
Lovelock's handbook \cite{Lovelock}.
\end{proof}

%%%%%%%%%%%%%%%%%%%%%%%%%%%%%%%%%%%%%%%%%%%%%%%%%%%%%%%%%%%%%%%%%%%%%%%%%%%%%%
%%%%%%%%%%%%%%%%%%%%%%%%%%%%%%%%%%%%%%%%%%%%%%%%%%%%%%%%%%%%%%%%%%%%%%%%%%%%%%

\section{Symmetric and Recurrent structures}
From now on, we restrict to Riemannian manifolds $({\mathcal M}^n,g)$. If
additional differential structures are present, the 
differential identities (\ref{iii}), (\ref{DdivR}) and (\ref{divVeblen}) 
simplify to interesting algebraic constraints.

A simple example is given by a {\em Locally Symmetric Space} \cite{KobNom}, 
i.e. a Riemannian manifold such that $\nabla_a  R_{bcd}{}^e = 0$. Then the 
aforementioned identities imply straightforwardly the algebraic ones 
\begin{eqnarray}
&& R_{(abc}{}^m R_{d)me}{}^f
-R_{ace}{}^m R_{bdm}{}^f + R_{acm}{}^f R_{bde}{}^m =0\label{RRR}\\
&& R_{am}R_{bce}{}^m + R_{bm}R_{cae}{}^m+ R_{cm}R_{abe}{}^m =0  \label{RR}\\
&& R_{am} R_{bec}{}^m 
-R_{bm}R_{ace}{}^m + R_{cm} R_{eba}{}^m -R_{em} R_{cab}{}^m =0\label{RRRR}
\end{eqnarray}
We show that these identities hold in several circumstances.
An example is a manifold with harmonic curvature \cite{Besse}, 
$\nabla_a R_{bcd}{}^a = 0$; in this less stringent case the general
property (\ref{DdivR}) yields (\ref{RR}) and (\ref{RRRR}). 
A slightly more general case is now considered
\begin{defin}
A manifold is Nearly Conformally Symmetric, (NCS)$_n$, 
(Roter \cite{Roter}) if
\begin{align}
\nabla_aR_{bc} - \nabla_b R_{ac} = \frac{1}{2(n-1)}[g_{bc}\nabla_a R -
g_{ac}\nabla_{b}R],\label{NCSM}
\end{align}
where $R=R_a{}^a$ is the curvature scalar. 
\end{defin}
Since $\nabla_a R_{bc}-
\nabla_b R_{ac}=-\nabla_m R_{abc}{}^m$, (NCS)$_n$ are a special case of
(\ref{divQ}) with $\nabla_m K_{bce}{}^m=0$ (trivial Codazzi tensor and
$\varphi = R$). Other particular cases are $K=0$ ($K$-flat) and 
$\nabla_a K_{bcd}{}^e=0$ 
($K$-symmetric). They yield, for the different choices of $K$, various types 
of $K-$flat/symmetric manifolds\cite{Singh99}:
conformally flat/symmetric ($K=C$)\cite{Chaki63,Dzerdinski78}, 
projectively flat/symmetric ($K=P$)\cite{Glodek71}, 
concircular or conharmonic symmetric\cite{Adati67}, 
and quasi conformally flat/symmetric.
Because of  Prop. \ref{propdivQ}, the following is true:
\begin{prop}\label{propNCS}
For (NCS)$_n$ manifolds, and for $K-$flat/symmetric manifolds, eqs. 
(\ref {RR}) and (\ref{RRRR}) hold.
\end{prop}

By weakening the defining condition of a Locally Symmetric Space, one
introduces a {\em Semisymmetric Space}: $[\nabla_a,\nabla_b] R_{cdef}=0$
\cite{Szabo82}.
\begin{prop}\label{semisymm}
For a Semisymmetric Space, eqs.(\ref{RRR}), (\ref{RR}) and (\ref{RRRR})
hold.
\end{prop}
\begin{proof}
First statement: eq.(\ref{commut}) simplifies
to $0=\nabla_{(a}\nabla_b R_{cd)ef}$; by eq.(\ref{iii}) the identity 
(\ref{RRR}) follows. Second statement: the definition implies a relation 
for the Ricci tensor: $[\nabla_a,\nabla_b] R_{ce}=0$. By inserting the 
identity $\nabla_m R_{abc}{}^m= \nabla_bR_{ac}-\nabla_a R_{bc}$
in the l.h.s. of eqs. (\ref{DdivR}) and (\ref{divVeblen}), those sides 
become sums of respectively three and four commutators 
of derivatives acting on Ricci tensors, and thus vanish. This implies 
eqs.(\ref{RR}) and (\ref{RRRR}).
\end{proof}

The algebraic property (\ref{RR}) holds in presence of even more general 
differential structures. 
\begin{defin} A manifold is Pseudosymmetric (Deszcz \cite{Deszcz87}) if:
\begin{align}
[\nabla_a, \nabla_b] R_{cdef} = L_R Q(g,R)_{cdefab} \label{Desz}
\end{align}
where $L_R$ is a scalar function and the Tachibana tensor is
\begin{eqnarray}
Q(g,R)_{cdefab}= 
-g_{cb}R_{adef}+g_{ca}R_{bdef}-g_{db}R_{caef}+g_{da}R_{cbef}\nonumber\\
-g_{eb}R_{cdaf}+g_{ea}R_{cdbf}-g_{fb}R_{cdea}+g_{fa}R_{cdeb}.
\end{eqnarray}
\end{defin}
\begin{thm}
For Pseudosymmetric Manifolds, the identities (\ref{RR}) and (\ref{RRRR})
hold. 
\end{thm}
\begin{proof}
The l.h.s. of eq.(\ref{RR}) can be written as a sum of commutators acting
on Ricci tensors:
\begin{eqnarray}
[\nabla_a,\nabla_c]R_{be}+[\nabla_b,\nabla_a]R_{ce}+[\nabla_c,\nabla_b]R_{ae}.
\end{eqnarray}
A commutator is obtained by contracting two indices in (\ref{Desz}); for
example, contraction of $c$ with $f$ gives
\begin{eqnarray}
[\nabla_a, \nabla_b] R_{de} = L_R (- g_{db}R_{ea}+g_{da}R_{eb}-
g_{eb}R_{da}+g_{ea}R_{db}),
\end{eqnarray}
i.e. the Ricci-pseudosymmetry property\cite{Deszcz89}.
Although each commutator is nonzero, their sum vanishes. Veblen's type 
identity is proven in a similar way. 
\end{proof}
We now show that (\ref{RRR}), (\ref{RR}) or (\ref{RRRR}) do hold in 
manifolds with recurrent structure.
\begin{defin}
A Riemannian manifold is a {\em Generalized Recurrent Manifold} if there
exist two vector fields $\lambda_a$ and $\mu_a$ such that
\begin{align}
\nabla_a R_{bcd}{}^e = \lambda_a R_{bcd}{}^e + \mu_a (\delta_b{}^e g_{cd}
-\delta_c{}^e g_{bd}) \label{grm}
\end{align}
\end{defin}
The manifolds were first introduced by Dubey\cite{Dubey79}, 
and studied by several authors\cite{De91,Maralabhavi99,Arslan09}. 
In particular, if $\mu_a=0$ the manifold is a {\em Recurrent Space}. 
Again, we shall prove that the algebraic identities (\ref{RRR}), (\ref{RR}) 
and (\ref{RRRR}) hold in this case. We need the following lemma, 
with a content slightly different than
the statement by Singh and Khan\cite{Singh00}.
\begin{lem}\label{XX}
In a Generalized Recurrent Manifold with curvature scalar $R\neq 0$
\begin{enumerate}
\item  if the curvature scalar $R$ is a constant then 
$\lambda $ is proportional to $\mu $ and either $\lambda $ is closed 
(i.e. $\nabla_a \lambda_b -\nabla_b \lambda_a=0$) 
or the manifold is a space of constant curvature,
$R_{abcd}=\frac{R}{n(n-1)}(g_{bd}g_{ac}-g_{ad}g_{bc})$; 
\item if the curvature scalar is not constant, then $\lambda $ is closed.
\end{enumerate}
\end{lem}
\begin{proof}
We need some relations that easily come from eq.(\ref{grm}): a) the 
contraction $a=e$ gives 
$\nabla_a R_{bcd}{}^a = \lambda_a R_{bcd}{}^a+\mu_b g_{cd}-\mu_c g_{bd}$.
A further divergence $\nabla^d$ gives zero in the l.h.s, by eq.(\ref{divdivR}),
and the r.h.s. in few steps is evaluated as 
\begin{align}
0=\frac{1}{2}[(\nabla_d\lambda_a)-(\nabla_a\lambda_d)]R_{bc}{}^{da}
-\mu_b\lambda_c+\mu_c\lambda_b +\nabla_c\mu_b -\nabla_b\mu_c; 
\label{***}
\end{align}
b) the contraction of $c=e$ in (\ref{grm}) yields
$\nabla_a R_{bd} = \lambda_a R_{bd} - (n-1)\mu_a g_{bd}$, and 
$\nabla_a R= \lambda_a R - n(n-1)\mu_a$; c) the commutator of covariant 
derivatives on the Riemann tensor of type (\ref{grm}) is
\begin{align}
&[\nabla_a, \nabla_b] R_{cde}{}^f = (\nabla_a\lambda_b-\nabla_b \lambda_a) 
R_{cde}{}^f\notag \\ 
& + (\delta_c{}^fg_{de}-\delta_d{}^fg_{ce})
(\nabla_a\mu_b-\nabla_b\mu_a -\lambda_a\mu_b +\lambda_b\mu_a)\label{commutator}
\end{align}
From b) we conclude that, 
if $\nabla_a R=0$, $\lambda$ and $\mu $ are collinear ($R$ is a number).\\
Then, eq. (\ref{***}) simplifies to 
\begin{eqnarray}
&0=\frac{1}{2}[(\nabla_d\lambda_a)-(\nabla_a\lambda_d)]R_{bc}{}^{da}
 +\frac{R}{n(n-1)}(\nabla_c\lambda_b -\nabla_b\lambda_c)\notag\\
&\quad=\frac{1}{2}[(\nabla_d\lambda_a)-(\nabla_a\lambda_d)][R_{bc}{}^{da}
 +\frac{R}{n(n-1)}\delta_{cb}^{da}]\equiv \frac{1}{2}A_{da}\tilde C_{bc}{}^{da}
\label{**}
\end{eqnarray}
($\tilde C$ is the (2,2) concircular tensor and $\delta_{cb}^{da}=\delta^a{}_b
\delta^d{}_c - \delta^a{}_c\delta^d{}_b$). Also eq.(\ref{commutator})
simplifies,
\begin{equation}
[\nabla_a, \nabla_b] R_{cde}{}^f = %(\nabla_a\lambda_b-\nabla_b \lambda_a) 
%[R_{cde}{}^f + \frac{R}{n(n-1)}(\delta_c{}^fg_{de}-\delta_d{}^fg_{ce})].
A_{ab}\tilde C_{cde}{}^f \label{23}
\end{equation}
Walker's identity (\ref{walker}) for the Riemann tensor (\ref{grm})
yields the algebraic relation
\begin{eqnarray}
0= A_{ab}\tilde C_{cdef}+A_{cd}\tilde C_{abef}+A_{ef}\tilde C_{abcd}\label{WAB}
\end{eqnarray}
Now Walker's lemma \cite{Walker51} is invoked: it implies that either
$A_{ab}=0$ or $\tilde C_{abcd}=0$. We give a proof based on (\ref{**}):
1) Saturate in eq.(\ref{WAB}) with $A^{ef}$ and use (\ref{**}): one gets
$A^{ef}A_{ef}\tilde C_{abcd}=0 \Rightarrow \tilde C_{abcd}=0$;
2) in the same way, by saturation with $\tilde C^{cdef}$ one gets 
$\tilde C_{abcd}\tilde C^{abcd}A_{ef}=0\Rightarrow A_{ef}=0 $.
Therefore either $\lambda $ is closed or the manifold is a space of 
constant curvature.

We now discuss the case $\nabla_a R\neq 0$. Take covariant derivative
$\nabla_b$ of $\nabla_a R = \lambda_a R - n(n-1)\mu_a$, and
exchange $a$ and $b$. Then
\begin{align} 
0=A_{ab}R+n(n-1)[\lambda_a\mu_b-\lambda_b\mu_a 
-\nabla_a\mu_b +\nabla_b\mu_a] \notag%\label{sigmalambda}
\end{align}
Enter
this in (\ref{***}),(\ref{commutator}), and get again (\ref{**}),(\ref{23}) 
where now $\tilde C\neq 0$.
The same procedure as above gives $A=0$, i.e. $\lambda $ is closed.
\end{proof}
\begin{thm}
In a Generalized Recurrent Manifold the properties (\ref{RRR}), 
(\ref{RR}) and (\ref{RRRR}) hold.
\end{thm}
\begin{proof}
If $\nabla R\neq 0$ then, by the previous Lemma \ref{XX}, $\lambda $ 
is always closed and, by eq.(\ref{23}), the space is semisymmetric. 
Then eqs.(\ref{RRR}),(\ref{RR}) and (\ref{RRRR}) hold by Prop.\ref{semisymm}.
\\ 
If $\nabla R=0$ then $\lambda $ and $\mu $ are collinear (Lemma \ref{XX}) and
eq.(\ref{23}) holds again. The Lemma states that either $\lambda $
is closed or the space has constant curvature. In both cases the manifold
is semisymmetric and (\ref{RRR}),(\ref{RR}),(\ref{RRRR}) hold.
\end{proof}

%%%%%%%%%%%%%%%%%%% K R M

The afore mentioned Recurrent structures are special cases of a new one, 
which we now define. It arises naturally from the invariance stated in
eq. (\ref{DdivQ}) stemming from Lovelock's identity.

\begin{defin}
A Riemannian manifold with a curvature tensor $K$ such that
eq.(\ref{divQ}) is true, is named {\em $K$-Recurrent Manifold} (KRM)
if $\nabla_a K_{bcd}{}^e =\lambda_a K_{bcd}{}^e $ where $\lambda $ is a nonzero
covector field.
\end{defin}

Therefore, KR-manifolds include known cases as Conformally-recurrent, 
Concircular-recurrent etc. (see \cite{Khan04} for a compendium).

In general, the Bianchi identity for a tensor $K$ contains a tensor source
$B$ (see Appendix for some relevant examples).
In a KRM it is $\lambda_{(a} K_{bc)d}{}^e = B_{abcd}{}^e$. When $\lambda $ is
closed, one obtains a remarkable property:

\begin{thm}\label{AA}
In a KRM with closed $\lambda $
\begin{align}
 R_{am}R_{bce}{}^m + R_{bm}R_{cae}{}^m + R_{cm}R_{abe}{}^m =\frac{1}{A}
\nabla_m B_{abce}{}^m
\end{align}
\end{thm}
\begin{proof}
$\nabla_a\nabla_m K_{bcd}{}^m =(\nabla_a\lambda_m) K_{bcd}{}^m+
\lambda_m \lambda_a K_{bcd}{}^m$.
Cyclic permutation on $(abc)$ and summation
$\nabla_a\nabla_m K_{bcd}{}^m+\nabla_b\nabla_m K_{cad}{}^m + 
\nabla_c\nabla_m K_{abd}{}^m = (\nabla_a\lambda_m) K_{bcd}{}^m+
(\nabla_b\lambda_m) K_{cad}{}^m+(\nabla_c\lambda_m) K_{abd}{}^m+
\lambda_m \lambda_a K_{bcd}{}^m +\lambda_m \lambda_b K_{cad}{}^m +
\lambda_m \lambda_c K_{abd}{}^m $
Evaluate $\nabla_m$ of Bianchi identity with $e=m$:
$(\nabla_m\lambda_{(a})K_{bc)d}{}^m+ \lambda_m\lambda_{(a}K_{bc)d}{}^m
= \nabla_m B_{abcd}{}^m $.
Use closure property and Lovelock's identity to conclude.
\end{proof}

\begin{cor}
For the tensors $K= C, P,\tilde C, N, W$ listed in the Appendix, 
the theorem \ref{AA} holds with null r.h.s.
\end{cor}
\begin{proof}
In the appendix one notes that $\nabla_m B_{abce}{}^m$ is either $0$ or a 
multiple of the l.h.s. (different from $A$).
\end{proof}

\begin{xrem} 
It is well known that \emph{Concircular Recurrency} is equivalent 
to \emph{Generalized Recurrency} \cite{Arslan09,Dubey79}.
\end{xrem}

%%%%%%%%%%%%%%%%%%%%%%%%%%%%%%%%%%%%%%%%%%%%%%%%%%%%%%%%%%%%%%%%%%%%
%%%%%%%%%%%%%%%%%%%%%%%%%%%%%%%%%%%%%%%%%%%%%%%%%%%%%%%%%%%%%%%%%%%%

\section{Weakly Ricci Symmetric Manifolds (WRS)$_n$ }\label{weaklysymm}
\begin{defin}
A (WRS)$_n$ is a Riemannian manifold with non-zero symmetric Ricci tensor
such that
\begin{align}
\nabla_a R_{bc} = A_a R_{bc} + B_b R_{ac} + D_c R_{ab} \label{WRS}
\end{align}
with A, B  and D are nonzero covector fields. 
\end{defin}
The manifolds were introduced by 
Tam\'assy and Binh \cite{Tamassy89}, and include the 
physically relevant Robertson-Walker space-times \cite{De02}, 
or the perfect fluid space-time \cite{De+Gosh04}. If $B=D$ the manifold is
{\em Ricci-Recurrent}. 
Most of the literature concentrates on the difference 
$B-D$, and prove that in (WRS)$_n$ that are conformally flat \cite{De03,De05} 
or quasi-conformally flat \cite{Jana07}, $B-D$ is a concircular vector.
We here show that Lovelock's identity (\ref{DdivR}) allows to discuss new
general properties of $A$, $B$, $D$.

\begin{lem}\label{A-B-D}
For $\alpha =A-B$ or $A-D$:
\begin{align} 
&R_{cb}(\nabla_d \alpha_a - \nabla_a \alpha_d) + 
R_{ca}(\nabla_b \alpha_d - \nabla_d \alpha_b) +
R_{cd}(\nabla_a \alpha_b - \nabla_b \alpha_a) \label{A-B}\\
&= R_{dm}R_{bac}{}^m + R_{bm}R_{adc}{}^m + R_{am}R_{dbc}{}^m
\notag
\end{align}
\end{lem}
\begin{proof}
From the definition of WRS$_n$ and the second Bianchi identity 
$\nabla_mR_{bac}{}^m =\nabla_a R_{bc} - \nabla_b R_{ac}$ one gets immediately 
$\nabla_m R_{bac}{}^m = \alpha_a R_{bc}-\alpha_b R_{ac}$, with 
$\alpha = A-B$. A further covariant derivative gives
\begin{align}
\nabla_d\nabla_mR_{bac}{}^m=(\nabla_d\alpha_a) R_{bc}-(\nabla_d\alpha_b) R_{ac}
+\alpha_a\nabla_dR_{bc}-\alpha_b\nabla_dR_{ac}\nonumber
\end{align}
Summation is done on cyclic permutation of $d,b,a$:
\begin{align}
&\nabla_d\nabla_mR_{bac}{}^m+\nabla_b\nabla_mR_{adc}{}^m+
\nabla_a\nabla_mR_{dbc}{}^m = \notag \\
&(\nabla_d\alpha_a -\nabla_a\alpha_d) R_{bc}+
(\nabla_b\alpha_d -\nabla_d\alpha_b) R_{ac}+
(\nabla_a\alpha_b -\nabla_b\alpha_a) R_{dc}\notag\\
&+\alpha_a(\nabla_dR_{bc}-\nabla_b R_{dc})+
\alpha_d(\nabla_bR_{ac}-\nabla_a R_{bc})+
\alpha_b(\nabla_aR_{dc}-\nabla_d R_{ac})\notag
\end{align}
The terms with derivatives of Ricci tensors vanish because
$\nabla_d R_{bc}-\nabla_b R_{dc}=\alpha_d R_{bc}-\alpha_b R_{dc}$. Then,
by eq.(\ref{DdivR}), we obtain (\ref{A-B}). The case $\alpha=A-D$ is
proven in the same way starting from the identity
$\nabla_mR_{abc}{}^m =\nabla_c R_{ba} - \nabla_a R_{bc}$.
\end{proof}

\begin{thm}
If $\det [R^a{}_b]\neq 0$ then $B=D$. 
\end{thm}
\begin{proof}
Case $\det R\neq 0$: because the Ricci tensor is symmetric, the antisymmetric 
part of eq.(\ref{WRS}) is: $0=(B-D)_b R_{ac}- (B-D)_c R_{ab}$. Left 
multiplication by $(R^{-1})^{ad}$ and summation on $a$ gives: 
$0=\delta^d{}_c (B-D)_b - \delta^d{}_b (B-D)_c $. Then put $d=b$ and sum: 
$0=(1-n)(B-D)_c \,\Rightarrow B=D.$
\end{proof}

\begin{xrem} If $\beta =B-D\neq 0$ then $R^a{}_b\beta^b=R\beta^a$, 
where $R$ is the nonzero scalar curvature\cite{De97}.
The validitity of Lemma \ref{A-B-D} for both $A-B$ and $A-D$ implies, 
by subtraction, an equation for $\beta $:
\begin{align} 
R_{cb}(\nabla_d \beta_a - \nabla_a \beta_d) + 
R_{ca}(\nabla_b \beta_d - \nabla_d \beta_b) +
R_{cd}(\nabla_a \beta_b - \nabla_b \beta_a) =0, 
\end{align}
and left multiplication by $\beta^c$ gives the differential identity
\begin{align} 
\beta_b(\nabla_d \beta_a - \nabla_a \beta_d) + 
\beta_a(\nabla_b \beta_d - \nabla_d \beta_b) +
\beta_d(\nabla_a \beta_b - \nabla_b \beta_a) =0. 
\end{align}
\end{xrem}

\begin{thm}\label{A-Bclosed}
In a (WRS)$_n$ manifold with nonsingular Ricci tensor, the covector
$A-B$ is closed iff 
\begin{align}
 R_{dm}R_{bac}{}^m + R_{bm}R_{adc}{}^m + R_{am}R_{dbc}{}^m =0 \label{WRS-RR}
\end{align}
\end{thm}
\begin{proof}
If $A-B$ (which equals $A-D$ because $\det R\neq 0$) is closed then
(\ref{WRS-RR}) holds because of the Lemma. If the r.h.s. of Lemma vanishes,
\begin{align} 
R_{cb}(\nabla_d \alpha_a - \nabla_a \alpha_d) + 
R_{ca}(\nabla_b \alpha_d - \nabla_d \alpha_b) +
R_{cd}(\nabla_a \alpha_b - \nabla_b \alpha_a) =0 \notag
\end{align}
the index $c$ is raised and multiplication is made by 
$(R^{-1})^s{}_c$:  
\begin{align} 
\delta^s{}_b(\nabla_d \alpha_a - \nabla_a \alpha_d) + 
\delta^s{}_a(\nabla_b \alpha_d - \nabla_d \alpha_b) +
\delta^s{}_d(\nabla_a \alpha_b - \nabla_b \alpha_a) =0 \notag
\end{align}
Put $s=b$ and sum: $(n-2)(\nabla_d \alpha_a - \nabla_a \alpha_d)
=0$. Then, if $n>2$, $\alpha $ is closed.
\end{proof}

(WRS)$_n$ manifolds of physical relevance that
fulfill the condition (\ref{WRS-RR}) are the conformally flat
WRS-manifolds, i.e. (WRS)$_n$ manifolds whose Weyl tensor 
(see Appendix) vanishes\cite{De02,De03}.

\begin{cor}
If a (WRS)$_n$ manifold is conformally flat and the Ricci matrix is 
nonsingular, then $A-B$ is closed.
\end{cor}
\begin{proof}
The divergence of the Weyl tensor (Appendix)
takes the form (\ref{divQ}), where the Codazzi 
tensor is $g_{ab}$. Because of the general proposition (\ref{DdivQ})
we have
\begin{align}
&\nabla_a\nabla_m C_{bdc}{}^m + \nabla_{b}\nabla_m C_{dac}{}^m+
\nabla_d\nabla_m C_{abc}{}^m\nonumber \\
&=\frac{n-3}{n-2}(
 R_{am}R_{bdc}{}^m + R_{bm}R_{dac}{}^m + R_{dm}R_{abc}{}^m ) \nonumber
\end{align}
If $n>3$ and if the Weyl tensor itself or its covariant divergence vanish,
we enter in the case of theorem \ref{A-Bclosed}
\end{proof}

\section{A wave equation for the Riemann tensor}
\begin{prop} For a Levi-Civita connection with $R_{ab}=0$,  
the contraction in \eqref{iii} 
with $g^{ab}$ yields Lichnerowicz's non 
linear wave equation (\ref{Lichnerowicz}).
\end{prop}
\begin{proof}
Since $\nabla_a g_{bc}=0$, indices can be lowered or raised freely under 
covariant derivation. The Riemann tensor gains the symmetry 
$R_{abcd}=R_{cdab}$ and the further condition $R_{ab}=0$ implies that 
$\nabla_k R^k{}_{abc}=0$. Eq.\ref{Lichnerowicz} follows immediately.
\end{proof}

\subsection*{Appendix}
We collect the useful formulae for the $K$-curvature tensors in a 
$n-$dimensional Riemannian manifold:
a) definition, b) divergence, c) cyclic sum of derivatives (unlike the 
II Bianchi identity, we get a nonzero tensor $B$), d) divergence of $B$
(the r.h.s. of c).\\

\noindent
\underline{Projective tensor}
\begin{eqnarray}
a)&& P_{bcd}{}^e = R_{bcd}{}^e + \frac{1}{n-1} (\delta^e{}_b R_{cd}-
\delta^e{}_c R_{bd} )\nonumber\\
b)&& \nabla_m P_{bcd}{}^m = \frac{n-2}{n-1}\nabla_m R_{bcd}{}^m
\nonumber\\
c)&&\nabla_a P_{bcd}{}^e +\nabla_b P_{cad}{}^e +\nabla_c P_{abd}{}^e=
\frac{1}{n-1}(\delta^e{}_a \nabla_p R_{bcd}{}^p + 
\delta^e{}_b \nabla_p R_{cad}{}^p + \delta^e{}_c \nabla_p R_{abd}{}^p )
\nonumber\\
d)&& \nabla_mB_{abcd}{}^m 
= \frac{1}{n-1}( 
\nabla_a\nabla_p R_{bcd}{}^p+\nabla_b\nabla_p R_{cad}{}^p
+\nabla_c\nabla_p R_{abd}{}^p)\nonumber
\end{eqnarray}
\underline{Conformal (Weyl) tensor}
\begin{eqnarray}
a)&& C_{abc}{}^d = R_{abc}{}^d +\frac{\delta_a{}^d R_{bc}-
\delta_b{}^d R_{ac} +R_a{}^d g_{bc} - R_b{}^d g_{ac}}{n-2}-
R \frac{\delta_a{}^d g_{bc} -\delta_b{}^d g_{ac} }{(n-1)(n-2)}\nonumber\\
b)&& \nabla_m C_{abc}{}^m = \frac{n-3}{n-2}\left [ \nabla_m R_{abc}{}^m +
\frac{1}{2(n-1)}(g_{bc}\nabla_a R - g_{ac}\nabla_b R)\right ]\nonumber\\
c)&&\nabla_a C_{bcd}{}^e +\nabla_b C_{cad}{}^e +\nabla_c C_{abd}{}^e=
\frac{1}{n-2}[\delta^e{}_a \nabla_p R_{bcd}{}^p + 
\delta^e{}_b \nabla_p R_{cad}{}^p + \delta^e{}_c \nabla_p R_{abd}{}^p
\nonumber\\ 
&&+g_{cd}(\nabla_a R_b{}^e - \nabla_b R_a{}^e)+g_{ad}
(\nabla_b R_c{}^e - \nabla_c R_b{}^e)+g_{bd}
(\nabla_c R_a{}^e - \nabla_a R_c{}^e)]-\frac{1}{(n-1)(n-2)}\nonumber\\
&&[\delta^e{}_a (g_{bd}\nabla_c R -g_{cd}\nabla_b R) +
\delta^e{}_b (g_{cd}\nabla_a R -g_{ad}\nabla_c R)+
\delta^e{}_c (g_{ad}\nabla_b R -g_{bd}\nabla_a R)]\nonumber\\ 
d)&& \nabla_mB_{abcd}{}^m 
= \frac{1}{n-2}( 
\nabla_a\nabla_p R_{bcd}{}^p+\nabla_b\nabla_p R_{cad}{}^p
+\nabla_c\nabla_p R_{abd}{}^p)\nonumber
\end{eqnarray}
\underline{Concircular tensor}
\begin{eqnarray}
a)&& \tilde C_{bcd}{}^e = R_{bcd}{}^e + \frac{R}{n(n-1)} (\delta^e{}_b g_{cd}-
\delta^e{}_c g_{bd} )\nonumber\\
b)&& \nabla_m \tilde C_{bcd}{}^m = \nabla_m R_{bcd}{}^m +\frac{1}{n(n-1)}
(\nabla_b R g_{cd}-\nabla_c R g_{bd})\nonumber\\
c)&& \nabla_a \tilde C_{bcd}{}^e +\nabla_b \tilde C_{cad}{}^e +
\nabla_c \tilde C_{abd}{}^e=
\frac{1}{n(n-1)}[\delta^e{}_a (\nabla_c R g_{bd}{} -\nabla_b R g_{cd})\nonumber\\
&&+\delta^e{}_b (\nabla_a R g_{cd}{} -\nabla_c R g_{ad})+
\delta^e{}_c (\nabla_b R g_{ad}{} -\nabla_a R g_{bd})]\nonumber\\
d)&& \nabla_m B_{abcd}{}^m = 0\nonumber
\end{eqnarray}
\underline{Conharmonic tensor}
\begin{eqnarray}
a)&& N_{bcd}{}^e = R_{bcd}{}^e + \frac{1}{n-2}[\delta_b{}^e R_{cd}
-\delta_c{}^e R_{bd}+ R_b{}^e g_{cd}- R_c{}^e g_{bd}]\nonumber\\
b)&& \nabla_m N_{bcd}{}^m = \frac{n-3}{n-2}\nabla_m R_{bcd}{}^m+\frac{1}{2(n-2)}
(\nabla_bR g_{cd}-\nabla_c R g_{bd})\nonumber\\
c)&& \nabla_a N_{bcd}{}^e +\nabla_b N_{cad}{}^e +\nabla_c N_{abd}{}^e =
\frac{1}{n-2}[\delta_a{}^e \nabla_p R_{bcd}{}^p+\delta_b{}^e \nabla_p R_{cad}{}^p+\delta_c{}^e \nabla_p R_{abd}{}^p\nonumber\\
&&+g_{cd}(\nabla_a R_b{}^e-\nabla R_{}^e)
+g_{ad}(\nabla_b R_c{}^e-\nabla_c R_b{}^e)
+g_{cd}(\nabla_c R_a{}^e-\nabla_aR_c{}^e)]
\nonumber\\
d)&&  \nabla_m B_{abcd}{}^m = \frac{1}{n-2}(\nabla_a\nabla_pR_{bcd}{}^p
+\nabla_b\nabla_pR_{cad}{}^p+\nabla_c\nabla_pR_{abd}{}^p)
\nonumber 
\end{eqnarray}
\underline{Quasi-conformal tensor}
\begin{eqnarray}
a)&& W_{bcd}{}^e = a \tilde C_{bcd}{}^e + b(n-2)[C_{bcd}{}^e -\tilde C_{bcd}{}^e]
\nonumber\\
b)&& \nabla_m W_{bcd}{}^m = (a+b)\nabla_m R_{bcd}{}^m + 
\frac{2a-b(n-1)(n-4)}{2n(n-1)}(\nabla_b R g_{cd} - \nabla_c R g_{bd})
\nonumber\\
c)&& \nabla_a W_{bcd}{}^e +\nabla_b W_{cad}{}^e + \nabla_c W_{abd}{}^e 
= -b(n-2)[\nabla_a C_{bcd}{}^e +\nabla_b C_{cad}{}^e+\nabla_c C_{abd}{}^e]
\nonumber\\
&&+[a+b(n-2)][\nabla_a \tilde C_{bcd}{}^e +\nabla_b \tilde C_{cad}{}^e
+\nabla_c \tilde C_{abd}{}^e]\nonumber\\
d)&& \nabla_m B_{abcd}{}^m  = -b (\nabla_a \nabla_p  R_{bcd}{}^p
+\nabla_b \nabla_p  R_{cad}{}^p + \nabla_c \nabla_p  R_{abd}{}^p) \nonumber
\end{eqnarray}

\end{document}